\documentclass{amsart}
\usepackage{amsxtra,amscd}
\usepackage{graphicx}
\usepackage{amsmath}
\usepackage{amsfonts}
\usepackage{amssymb}

\newtheorem{theorem}{Theorem}[section]
\newtheorem{lemma}[theorem]{Lemma}
\theoremstyle{definition}
\newtheorem{definition}[theorem]{Definition}

\newtheorem{proposition}[theorem]{Proposition}
\newtheorem{remark}[theorem]{Remark}

\numberwithin{equation}{section}

\begin{document}
\title{On the unramified extension of an arithmetic function field in several variables}
\author{Feng-Wen An}
\address{School of Mathematics and Statistics, Wuhan University, Wuhan,
Hubei 430072, People's Republic of China}
\email{fwan@amss.ac.cn}

\begin{abstract}
In this paper we will give a scheme-theoretic discussion on the unramified extensions of an arithmetic function field in several variables. The notion of unramified discussed here is parallel to that in algebraic number theory and for the case of classical varieties, coincides with that in Lang's theory of unramified class fields of a function field in several variables. It is twofold for us to introduce the notion of unramified. One is for the computation of the \'{e}tale fundamental group of an arithmetic scheme; the other is for  an ideal-theoretic theory of unramified class fields over an arithmetic function field in several variables. Fortunately, in the paper we will also have operations on unramified extensions such as base changes, composites, subfields, transitivity, etc. It will be proved that a purely transcendental extension over the rational field has a trivial unramified extension. As an application, it will be seen that the affine scheme of a ring  over the ring of integers in several variables has a trivial \'{e}tale fundamental group.
\end{abstract}

\maketitle

\begin{center}
{\tiny {Contents} }
\end{center}

{\tiny \qquad {1. Introduction} }

{\tiny \qquad {2. Basic Definitions} }

{\tiny \qquad {3. Operations on Unramified Extensions}}

{\tiny \qquad {4. Unramified Extensions of a Purely Transcendental Extension}}

{\tiny \qquad {5. An Application to the \'{E}tale Fundamental Group}}

{\tiny \qquad {References}}{}

\section{Introduction}

In this paper we will discuss the unramified extensions of an arithmetic function field, i.e., a function field over $\mathbb{Z}$ in several variables. When restricted to classical varieties, the notion of unramified discussed here coincides with that in Lang's theory of unramified class fields of a function field in several variables (see \cite{lang}).

In deed, it is twofold for us to introduce the notion of unramified extensions of an arithmetic function field in a scheme-theoretic manner. One is for the computation of the \'{e}tale fundamental group of an arithmetic scheme (see \S 5 in the paper or see \cite{an2,an4} for a general case). The other is for one to try to have an ideal-theoretic theory of unramified class fields over an arithmetic function field in several variables (for instance, see \cite{an5}).

The notion of unramified for arithmetic function fields discussed in the paper is parallel to that in algebraic number theory (for instance, see \cite{neu}). In fact, let $K$ be an arithmetic function field, i.e., $K$ is the fractional field of a ring of the form $$
\mathbb{Z}[t_{1},t_{2},\cdots, t_{n},s_{1},s_{2},\cdots, s_{m}]
$$
where
$
t_{1},t_{2},\cdots, t_{n}
$
are algebraic independent variables over $\mathbb{Q}$ and
$
s_{1},s_{2},\cdots, s_{m}
$
are algebraic over the field $\mathbb{Q}[t_{1},t_{2},\cdots, t_{n}]$. As a counterpart in algebraic number theory, the ring of algebraic integers of the field $K$, denoted by $\mathcal{O}_{K}$, is also defined to be the set of elements $x\in K$ that are integral over the ring $\mathbb{Z}[t_{1},t_{2},\cdots, t_{n}]$.

Let $L\supseteq K$ be an arithmetic function field also over variables $t_{1},t_{2},\cdots, t_{n}$. Then $L$ is unramified over $K$ if the scheme $Spec(\mathcal{O}_{L})$ is unramified over $Spec(\mathcal{O}_{K})$ by the morphism induced from the inclusion map $K\hookrightarrow L$. It will be seen that the unramified is independent of the choice of variables $t_{1},t_{2},\cdots, t_{n}$. See \S 2 for details.

Fortunately, there also exist the same operations on unramified extensions of arithmetic function fields as  in algebraic number theory (see \cite{neu}), such as base changes, composites, subfields, transitivity, etc. See \S 3 for details.

In \S 4 we will prove that a purely transcendental extension $\mathbb{Q}(t_{1},t_{2},\cdots, t_{n})$ has a trivial unramified extension. As an immediate application, in \S 5 it will be seen that $Spec(\mathbb{Z}[t_{1},t_{2},\cdots, t_{n}])$ has a trivial \'{e}tale fundamental group, i.e., $$\pi_{1}^{et}(Spec(\mathbb{Z}[t_{1},t_{2},\cdots, t_{n}]))=\{0\}.$$

\subsection*{Acknowledgment}

\quad The author would like to express his sincere gratitude to Professor Li
Banghe for his advice and instructions on algebraic geometry and topology.

\section{Basic Definitions}

Let's recall basic definitions for arithmetically unramified extensions of a function field over $\mathbb{Z}$ in several variables that are introduced in
\cite{an4}, where the maximal arithmetically unramified extension will be used to give a computation of the \'{e}tale fundamental group of an arithmetic
scheme (see \S 5).

\subsection{Convention}

In the paper, an \textbf{arithmetic function field} \textbf{over variables} $%
t_{1},t_{2},\cdots, t_{n}$ is the fractional field $K=Fr(D)$ of an integral
domain
\begin{equation*}
D=\mathbb{Z}[t_{1},t_{2},\cdots, t_{n},s_{1},s_{2},\cdots, s_{m}]
\end{equation*}
of finite type over $\mathbb{Z}$, where
\begin{equation*}
s_{1},s_{2},\cdots, s_{m}
\end{equation*}
are algebraic over $\mathbb{Q}[t_{1},t_{2},\cdots, t_{n}]$ and
\begin{equation*}
t_{1},t_{2},\cdots, t_{n},
\end{equation*}
are (algebraic independent) variables over $\mathbb{Q}$.

\subsection{Ring of algebraic integers in an arithmetic function field}

Let $K$ be an arithmetic function field over variables $t_{1},t_{2},\cdots,
t_{n}$. As a counterpart in algebraic number theory, put
\begin{itemize}
\item $\mathcal{O}_{[t_{1},t_{2},\cdots, t_{n}]K}\triangleq$ the subring of
elements $x\in K$ that are integral over the integral domain $\mathbb{Z}%
[t_{1},t_{2},\cdots, t_{n}]$.
\end{itemize}

The ring $\mathcal{O}_{[t_{1},t_{2},\cdots, t_{n}]K}$ is said to be the
\textbf{ring of algebraic integers} of $K$ \textbf{over variables} $%
t_{1},t_{2},\cdots, t_{n}$. Sometimes, we will write $$\mathcal{O}_{K}=
\mathcal{O}_{[t_{1},t_{2},\cdots, t_{n}]K}$$ for brevity.

\begin{proposition}
We have $K=Fr(\mathcal{O}_{K})$.
\end{proposition}

\begin{proof}
Let $F=\mathbb{Q}(t_{1},t_{2},\cdots, t_{n})$. Take an $s=s_{j}$ with $1\leq
j\leq m$ in the above. As $s$ is algebraic over $F$, there is a polynomial $%
f(X)\in F[X]$ such that $f(s)=0$. Without loss of generality, suppose
\begin{equation*}
f(X)=a_{n}X^{n}+a_{n-1}X^{n-1}+\cdots+a_{1}X+a_{0}
\end{equation*}
such that each coefficient $a_{i}$ of $f(X)$ is contained in $\mathbb{Z}%
[t_{1},t_{2},\cdots, t_{n}]$.

Let $a_{n}\not=0$. We have
\begin{equation*}
(a_{n}s)^{n}+a_{n-1}(a_{n}s)^{n-1}+\cdots+a_{0}a_{n}^{n-1}=0.
\end{equation*}
Then $a_{n}s\in \mathcal{O}_{K}$; hence, $K=Fr(\mathcal{O}_{K})$.
\end{proof}

\subsection{Arithmetically unramified}

Let  $K$ and $L$ be two arithmetic function fields over variables $t_{1},t_{2},\cdots, t_{n}$. We say
\begin{equation*}
(L,\mathcal{O}_{L})\supseteq (K,\mathcal{O}_{K})
\end{equation*}
if and only if there are
\begin{equation*}
\overline{K}\supseteq L\supseteq K;
\end{equation*}
\begin{equation*}
\mathcal{O}_{L}\supseteq \mathcal{O}_{K}\supseteq \mathbb{Z}
[t_{1},t_{2},\cdots, t_{n}].
\end{equation*}

\begin{definition}
Let $L/ K$ be arithmetic function fields over variables $t_{1},t_{2},\cdots,
t_{n}$. $L$ is  \textbf{arithmetically unramified} over $K$ (\textbf{%
relative to variables} $t_{1},t_{2},\cdots, t_{n}$) if for the pairs $(L,%
\mathcal{O}_{L})\supseteq (K, \mathcal{O}_{K})$, the affine scheme $Spec(%
\mathcal{O}_{L})$ is unramified over $Spec(\mathcal{O}_{K})$ by the morphism
induced from the inclusion map $K\hookrightarrow L $.
\end{definition}

It will be seen that the arithmetically unramified extension $L/K$ is
independent of the choice of variables $t_{1},t_{2},\cdots, t_{n}$ (see
\emph{Remark 3.6} below).

\begin{remark}
These arithmetically unramified extensions coincide exactly with those in
algebraic number theory (for instance, see \cite{mln,neu}).
\end{remark}

\section{Operations on Unramified Extensions}

As in algebraic number theory, in this section we will have several operations on arithmetically unramified
extensions. It will be seen that such unramified extensions are transitive and that
the base changes, the composites, and the subfields are still
arithmetically unramified. A part of the results in the section are discussed in \cite{an4} in an
not obvious manner.

Here, the approach to arithmetically unramified, which is in a scheme-theoretic manner, is still valid for number fields;
at the same time, in algebraic number theory, these results on number fields are obtained essentially from \emph{Hemsel's Lemma} (for instance, see \cite{neu}).

\subsection{Subfields and transitivity}

Let $K\subseteq L\subseteq M$ be arithmetic function fields over variables $%
t_{1},t_{2},\cdots, t_{n}$.

\begin{lemma}
Let $L/K$ and $M/L$ be arithmetically unramified, respectively. Then so is $%
M/L$.
\end{lemma}

\begin{proof}
It is immediate from the composite of two unramified morphisms is also
unramified (see \cite{sga1,mln}).
\end{proof}

\begin{lemma}
Let $L$ be arithmetically unramified over $K$. Take a subfield $L\subseteq
L_{0}\subseteq K$. Then $L/L_{0}$ and $L_{0}/K$ are arithmetically
unramified, respectively.
\end{lemma}

\begin{proof}
For the rings of algebraic integers, we have
\begin{equation*}
\mathcal{O}_{K}\subseteq \mathcal{O}_{L_{0}}\subseteq \mathcal{O}_{L};
\end{equation*}
\begin{equation*}
\mathcal{O}_{L_{0}}=\mathcal{O}_{L}\bigcap L_{0};
\end{equation*}
\begin{equation*}
\mathcal{O}_{K}=\mathcal{O}_{L}\bigcap K=\mathcal{O}_{L_{0}}\bigcap K.
\end{equation*}

Fixed any prime ideal $\mathfrak{P}$ of the ring $\mathcal{O}_{L}$. Put
\begin{equation*}
\mathfrak{p}=\mathfrak{P}\bigcap \mathcal{O}_{K};
\end{equation*}
\begin{equation*}
\mathfrak{P}_{0}=\mathfrak{P}\bigcap \mathcal{O}_{L_{0}}.
\end{equation*}
Then $\mathfrak{p}$ and $\mathfrak{P}_{0}$ are prime ideals of $\mathcal{O}%
_{K}$ and $\mathcal{O}_{L_{0}}$, respectively.

By the assumption that $L/K$ is arithmetically unramified, it is seen that
\begin{equation*}
\mathfrak{P}=\mathfrak{p}\mathcal{O}_{L}
\end{equation*}
holds and that
\begin{equation*}
\frac{\mathcal{O}_{L}/\mathfrak{P}}{\mathcal{O}_{K}/\mathfrak{p}}
\end{equation*}
is a finite separable extension (see \cite{f-k,sga1,mln}). It follows that
\begin{equation*}
\mathfrak{P}_{0}=\mathfrak{p}\mathcal{O}_{L_{0}}
\end{equation*}
and
\begin{equation*}
\mathfrak{P}=\mathfrak{P}_{0}\mathcal{O}_{L}
\end{equation*}
holds and that
\begin{equation*}
\frac{\mathcal{O}_{L}/\mathfrak{P}}{\mathcal{O}_{L_{0}}/\mathfrak{P}_{0}}
\end{equation*}
and
\begin{equation*}
\frac{\mathcal{O}_{L_{0}}/\mathfrak{P}_{0}}{\mathcal{O}_{K}/\mathfrak{p}}
\end{equation*}
are finite separable extensions, respectively.

Hence,
\begin{equation*}
Spec(\mathcal{O}_{L})/Spec(\mathcal{O}_{L_{0}})
\end{equation*}
and
\begin{equation*}
Spec(\mathcal{O}_{L_{0}})/Spec(\mathcal{O}_{K})
\end{equation*}
are unramified, respectively. This completes the proof.
\end{proof}

\subsection{Composites}

Let $K$, $L_{1}$, and $L_{2}$ be arithmetic function fields over several variables $
t_{1},t_{2},\cdots, t_{n}$.

\begin{lemma}
Suppose that $L_{1}/K$ and $L_{2}/K$ are
arithmetically unramified. Then the composite $L_{1}\cdot L_{2}$ of fields is also arithmetically unramified over $K$.
\end{lemma}

\begin{proof}
Consider the three pairs
\begin{equation*}
(K,\mathcal{O}_{K})\subseteq (L_{1},\mathcal{O}_{L_{1}})\bigcap (L_{2},
\mathcal{O}_{L_{2}})
\end{equation*}
of rings of algebraic integers, where
\begin{equation*}
\mathcal{O}_{L_{1}}=\mathcal{O}_{K}[a_{1},a_{2},\cdots, a_{m}];
\end{equation*}
\begin{equation*}
\mathcal{O}_{L_{2}}=\mathcal{O}_{K}[b_{1},b_{2},\cdots, b_{n}].
\end{equation*}

Denote by
\begin{equation*}
\mathcal{O}_{L_{1}}\cdot \mathcal{O}_{L_{2}}
\end{equation*}
the smallest subring of the field ${L_{1}}\cdot {L_{2}}$ that contains the
set
\begin{equation*}
\{x\cdot y\mid x\in \mathcal{O}_{L_{1}},y\in \mathcal{O}_{L_{2}}\}.
\end{equation*}

Suppose that
\begin{equation*}
\phi:\mathcal{O}_{L_{1}}\otimes _{\mathcal{O}_{K}} \mathcal{O} _{L_{2}}\to
\mathcal{O}_{L_{1}}\cdot \mathcal{O}_{L_{2}}
\end{equation*}
is the homomorphism of rings given by
\begin{equation*}
x_{1}\otimes x_{2}\mapsto x_{1}\cdot x_{2}.
\end{equation*}

Put
\begin{equation*}
\mathcal{O}=\phi(\mathcal{O}_{L_{1}}\otimes _{\mathcal{O}_{K}} \mathcal{O}
_{L_{2}}).
\end{equation*}

By properties of integral closures (see \cite{bourbaki}), we have
\begin{equation*}
\mathcal{O}_{L_{1}}\cdot \mathcal{O}_{L_{2}}\subseteq\mathcal{O}_{L_{1}\cdot
L_{2}};
\end{equation*}
\begin{equation*}
\mathcal{O}_{L_{1}\cdot L_{2}}\subseteq\mathcal{O}_{K}[a_{1},a_{2},\cdots,
a_{m},b_{1},b_{2},\cdots, b_{n}];
\end{equation*}
\begin{equation*}
\mathcal{O}_{K}[a_{1},a_{2},\cdots, a_{m},b_{1},b_{2},\cdots,
b_{n}]\subseteq \mathcal{O}_{L_{1}}\cdot \mathcal{O}_{L_{2}};
\end{equation*}
\begin{equation*}
\mathcal{O}_{L_{1}}\cdot \mathcal{O}_{L_{2}}\subseteq\mathcal{O}\subseteq
\mathcal{O}_{K}[a_{1},a_{2},\cdots, a_{m},b_{1},b_{2},\cdots, b_{n}].
\end{equation*}

Then
\begin{equation*}
\mathcal{O}=\mathcal{O}_{L_{1}\cdot L_{2}}=\mathcal{O}_{L_{1}} \cdot
\mathcal{O}_{L_{2}}
\end{equation*}
is the ring of algebraic integers in the field $L_{1}\cdot L_{2}$.

Hence, $\phi$ is a surjection. It is follows that $Spec(\mathcal{O}%
_{L_{1}\cdot L_{2}})$ is a closed immersion of the affine scheme $Spec(%
\mathcal{O}_{L_{1}}\otimes _{\mathcal{O}_{K}} \mathcal{O}_{L_{2}})$.

Now we have a tower of morphisms of schemes
\begin{equation*}
\begin{array}{l}
Spec(\mathcal{O}_{L_{1}\cdot L_{2}}) \\
\to Spec(\mathcal{O}_{L_{1}}\otimes _{\mathcal{O}_{K}} \mathcal{O} _{L_{2}})
\\
\cong Spec(\mathcal{O}_{L_{1}})\times _{Spec(\mathcal{O} _{K})} Spec(%
\mathcal{O}_{L_{2}}) \\
\to Spec(\mathcal{O}_{L_{1}}) \\
\to Spec(\mathcal{O}_{K}).%
\end{array}%
\end{equation*}

It is seen that $Spec(\mathcal{O}_{L_{1}\cdot L_{2}})$ is unramified over $
Spec(\mathcal{O} _{K}) $ from the base change of a unramified morphism (see
\cite{sga1}). Hence, the composite $L_{1}\cdot L_{2}$ is arithmetically
unramified over $K$.
\end{proof}

\subsection{Base changes}

Let $K$ and $L$ be two arithmetic function fields over several variables $
t_{1},t_{2},\cdots, t_{n}$.

\begin{lemma}
Let $K^{\prime}\supseteq K$ be an arithmetic function field over  $
t_{1},t_{2},\cdots, t_{n}$ and let $L/K$ be arithmetically unramified. Then the composite $L^{\prime}=L\cdot K^{\prime}$ of fields is
also arithmetically unramified over $K^{\prime}$.
\end{lemma}

\begin{proof}
Take the homomorphism
\begin{equation*}
\phi:\mathcal{O}_{L}\otimes_{\mathcal{O}_{K}}\mathcal{O}_{K^{\prime}}\to
\mathcal{O}_{L}\cdot \mathcal{O}_{K^{\prime}}
\end{equation*}
of rings given by
\begin{equation*}
x_{1}\otimes x_{2}\mapsto x_{1}\cdot x_{2}.
\end{equation*}
It is clear that $\phi$ is a surjection. Put
\begin{equation*}
\mathcal{O}_{L^{\prime}}=\phi(\mathcal{O}_{L}\otimes _{\mathcal{O}_{K}}
\mathcal{O} _{K^{\prime}}).
\end{equation*}

Consider the morphisms
\begin{equation*}
Spec(\mathcal{O}_{L^{\prime}})\to Spec(\mathcal{O}_{L})\times _{Spec(%
\mathcal{O} _{K})} Spec(\mathcal{O}_{K^{\prime}})\to Spec( \mathcal{O}%
_{K^{\prime}})
\end{equation*}
as we have done in \emph{Lemma 3.3}. It is seen that the composite of the two morphisms
above is unramified.
\end{proof}

\subsection{The unramified is independent of the choice of variables}

Let $K$ and $ L$ be arithmetic function fields over variables $%
t_{1},t_{2},\cdots, t_{n}$.

\begin{lemma}
Let $K$ be an arithmetic function field also over variables $%
t_{1}^{\prime},t_{2}^{\prime},\cdots, t_{n}^{\prime}$. Then $L$ is
arithmetically unramified over $K$ relative to variables $%
t_{1},t_{2},\cdots, t_{n}$ if and only if relative to variables $%
t_{1}^{\prime},t_{2}^{\prime},\cdots, t_{n}^{\prime}$.
\end{lemma}

\begin{proof}
Assume that $L$ is arithmetically unramified over $K$ relative to variables $%
t_{1},t_{2},\cdots, t_{n}$. Denote by $\mathcal{O}_{K}^{\prime}$ the subring
of elements $x\in K$ that are integral over the integral domain $\mathbb{Z}%
[t_{1}^{\prime},t_{2}^{\prime},\cdots, t_{n}^{\prime}]$.

Consider the morphisms
\begin{equation*}
Spec(\mathcal{O}_{L})\to Spec(\mathcal{O}_{L})\times _{Spec(\mathcal{O}
_{K})} Spec(\mathcal{O}_{K}^{\prime})\to Spec( \mathcal{O}_{K}^{\prime}).
\end{equation*}
as we have done in \emph{Lemma 3.3}. By the composite of the two morphisms
above, it is seen that $Spec(\mathcal{O}_{L})$ is unramified over $Spec(
\mathcal{O}_{K}^{\prime})$.
\end{proof}

\begin{remark}
Let $L$ be arithmetically unramified over $K$. Then it is independent of the
choice of variables $t_{1},t_{2},\cdots, t_{n}$
of $L$ over $K$ according to \emph{Lemma 3.5}.
\end{remark}

\subsection{The maximal unramified extension}

Now take the maximal arithmetically unramified extension of an arithmetic
function field as one does in algebraic number theory. It is twofold. One is for the computation
of the \'{e}tale fundamental group of an arithmetic scheme (see \cite{an4}); the other is for unramified
class fields of an arithmetic function field in an ideal-theoretic manner (see \cite{an5}).

Let $L$ be an arithmetic function field over several variables $t_{1},t_{2},\cdots,
t_{n}$. Fixed an  algebraic closure $\Omega$ of $L$.
Put
\begin{itemize}
\item $\left[ L \right](\Omega)^{au}\triangleq$ the set of all finite
arithmetically unramified subextensions of $L$ contained in $\Omega$.

\item $(L;\Omega)^{au}\triangleq {{\lim}_{\rightarrow_{Z\in {\left[ L \right]%
(\Omega)^{au}}}}}{\ \lambda_{Z}(Z)}$, i.e., the direct limit of the direct
system of rings indexed by $\left[ L \right](\Omega)^{au}$, where each $%
\lambda_{Z}:Z\in {\left[ L \right](\Omega)^{au}}\to \Omega$ is the $L$%
-embedding.
\end{itemize}

Here, $\left[ L;\Omega \right]^{au}$ is taken as a directed set by set
inclusion. By \emph{Remark 3.6} it is seen that $(L;\Omega)^{au}$ is
well-defined.

The subfield $(L;\Omega)^{au}$ of $\Omega$ is called the \textbf{maximal
arithmetically unramified extension} of $L$ in $\Omega$.

\begin{proposition}
Let $\left[ L \right](\Omega)^{au}_{0}$ be the set of all finite
arithmetically unramified Galois subextensions of $L$ contained in $\Omega$.
Then we have
\begin{equation*}
(L;\Omega)^{au}= {{\lim}_{\rightarrow_{Z\in {\left[ L \right]%
(\Omega)^{au}_{0}}}}}{\ \lambda_{Z}(Z)}.
\end{equation*}
\end{proposition}

\begin{proof}
It is immediate from the preliminary fact that $\left[ L \right]%
(\Omega)^{au}_{0}$ is a cofinal directed subset of $\left[ L \right]%
(\Omega)^{au}$.
\end{proof}

\begin{lemma}
The maximal arithmetically unramified extension of an arithmetic function field is an algebraic Galois extension, which is unique upon isomorphisms of fields. That is,
$(L;\Omega)^{au}$ is an algebraic Galois extension of $L$ and there is a natural isomorphism $$(L;\Omega)^{au}\cong (L;\Omega^{\prime})^{au}$$ between fields over $L$ if $\Omega^{\prime}$ is another algebraic closure of $L$.
\end{lemma}

\begin{proof}
It is immediate from \emph{Proposition 3.7}.
\end{proof}

\section{Unramified Extensions of a Purely Transcendental Extension}

As well-known, the rational field $\mathbb{Q}$ has a trivial unramified extension (see \cite{neu}). In this section we will prove that $\mathbb{Q}(t_{1},t_{2},\cdots,t_{n})$ also has a trivial arithmetically unramified extension.
As an immediate application, in \S 5 it will be sen that the \'{e}tale fundamental group of the affine scheme $Spec(\mathbb{Z}[t_{1},t_{2},\cdots,t_{n}])$ is trivial.

\subsection{Statement of the theorem}

Fixed a finite number of 
variables $t_{1},t_{2},\cdots,t_{n}$ over $\mathbb{Q}$. For the
unramified extensions, we have the following result.

\begin{theorem}
A purely transcendental extension $\mathbb{Q}(t_{1},t_{2},\cdots,t_{n})$ over $\mathbb{Q}$ has a trivial
 arithmetically unramified
extension. That is, we have
$$
\mathbb{Q}(t_{1},t_{2},\cdots,t_{n})
=\mathbb{Q}(t_{1},t_{2},\cdots,t_{n})^{au}.
$$
\end{theorem}

\begin{remark}
Let $t_{1},t_{2},\cdots,t_{n},\cdots$ be an infinite number of (algebraically independent) variables over $\mathbb{Q}$. By \emph{Theorem 4.1} it is seen that $$
\mathbb{Q}(t_{1},t_{2},\cdots,t_{n},\cdots)
=\mathbb{Q}(t_{1},t_{2},\cdots,t_{n},\cdots)^{au}.
$$
\end{remark}

\subsection{Proof of Theorem 4.1}

Now we prove \emph{Theorem 4.1}.

\begin{proof}
\textbf{(Proof of Theorem 4.1)} Let $K$ be an arithmetically unramified extension of $\mathbb{Q}(t_{1},t_{2},\cdots,t_{n})$. We will proceed in several steps.

\emph{Step 1.} For the ring of algebraic integers, it is seen that
\begin{equation*}
\mathcal{O}_{\mathbb{Q}(t_{1},t_{2},\cdots,t_{n})}=\mathbb{Z}[t_{1},t_{2},\cdots,t_{n}]
\end{equation*}
holds  according to preliminary facts on
integral closures (see \cite{bourbaki}).

For a maximal ideal $\mathfrak{P}$ of the ring $\mathcal{O}_{K}$ of
algebraic integers of $K$, the ideal
\begin{equation*}
\mathfrak{p}=\mathfrak{P}\bigcap \mathbb{Z}[t_{1},t_{2},\cdots,t_{n}]
\end{equation*}
must be a maximal ideal of $\mathbb{Z}[t_{1},t_{2},\cdots,t_{n}]$ since $\mathfrak{p}$ generates
$\mathfrak{P}$ in the localizations of the rings. Then it is seen that $Spec(\mathcal{%
O}_{K})/Spec(\mathbb{Z}[t_{1},t_{2},\cdots,t_{n}])$ is faithfully flat and hence \'{e}tale in
virtue of preliminary facts on \'{e}tale morphisms (see \cite{sga1,mln}).

\emph{Step 2.} There are a finite number of elements
\begin{equation*}
\mu,\cdots,\nu\in \overline{\mathbb{Q}(t_{1},t_{2},\cdots,t_{n})}
\end{equation*}
generating $\mathcal{O}_{K}$ over $\mathbb{Z}[t_{1},t_{2},\cdots,t_{n}]$, i.e.,
\begin{equation*}
\mathcal{O}_{K}=\mathbb{Z}[t_{1},t_{2},\cdots,t_{n}][\mu_{1},\cdots,\mu_{m}].
\end{equation*}

By \emph{Lemmas 5.4-5} suppose
$$
K=\mathbb{Q}(t_{1},t_{2},\cdots,t_{n})[\mu]
$$ and
$$\mathcal{O}_{K}=\mathbb{Z}[t_{1},t_{2},\cdots,t_{n}][\mu]$$
without loss of generality, where $\mu\triangleq \mu_{1}$ is a non-unit in $\mathcal{O}_{K}$.

\emph{Step 3.} Prove that either $\mu\in \mathbb{Z}[t_{1},t_{2},\cdots,t_{n}]$ holds or $\mu$ is an irreducible element in $\mathcal{O}_{K}$.

In deed, assume $\mu=\alpha \cdot\beta$ for some $\alpha,\beta\in \mathcal{O}_{K}$ such that $\alpha$ is irreducible in $\mathcal{O}_{K}$. We have $$(\mu)\subseteq (\alpha)$$ as ideals in $\mathcal{O}_{K}$. Then there is some $c_{1}\in \mathcal{O}_{K}$ such that $$\mu=c_{1}\cdot \alpha.$$

On the other hand, as $\alpha\in \mathcal{O}_{K}$, there is a polynomial $$\phi(X)=X^{n}+a_{n-1}X^{n-1}+\cdots+a_{1}X+a_{0}$$ with coefficients in the ring $\mathbb{Z}[t_{1},t_{2},\cdots,t_{n}]$ such that $$\alpha=\phi(\mu).$$

We have
\begin{equation*}
\begin{array}{l}
\mu=c_{1}\cdot \phi(\mu)\\

=c_{1}\cdot (\mu^{n}+a_{n-1}\mu^{n-1}+\cdots+a_{1}\mu+a_{0})
\end{array}
\end{equation*}

Then $$\mu\mid a_{0}$$ or $$\mu\mid c_{1}.$$

If $\mu\mid a_{0}$, we have $$\mu\in \mathbb{Z}[t_{1},t_{2},\cdots,t_{n}];$$ moreover, in such a case, we will complete the proof.

If $\mu\mid c_{1}$, put  $$c_{1}=c_{2}\cdot \mu$$ with $c_{2}\in\mathcal{O}_{K};$ then $$\mu=c_{2}\cdot \mu\cdot\alpha;$$
$$\alpha=\pm 1,$$
where there will be a contradiction; hence, $\mu$ is irreducible in $\mathcal{O}_{K}$.

\emph{Step 4.} Suppose $\mu$ is an irreducible element in $\mathcal{O}_{K}$. By \emph{Going-Down Theorem} (see \cite{bourbaki}), we have the
prime ideals
\begin{equation*}
\mathfrak{p}=(t_{1},t_{2},\cdots,t_{n},p)
\end{equation*}
and
\begin{equation*}
\mathfrak{P}=(t_{1},t_{2},\cdots,t_{n},\mu)
\end{equation*}
of $A$ and $B$, respectively, satisfying
\begin{equation*}
\mathfrak{p}=\mathfrak{P}\bigcap \mathbb{Z}[t_{1},t_{2},\cdots,t_{n}],
\end{equation*}
where $p \in \mathbb{N}$ is a prime number. This is due to the preliminary fact that the heights of the prime ideals $\mathfrak{p}$ and $\mathfrak{P}$ are equal.

Consider the localizations of the rings
\begin{equation*}
B\triangleq(\mathcal{O}_{K}\setminus \mathfrak{P})^{-1}\mathcal{O}_{K};
\end{equation*}
\begin{equation*}
A\triangleq(\mathbb{Z}[t_{1},t_{2},\cdots,t_{n}]\setminus \mathfrak{p})^{-1}\mathbb{Z}[t_{1},t_{2},\cdots,t_{n}]\cong\mathbb{F}_{p}[t_{1},t_{2},\cdots,t_{n}].
\end{equation*}
It is clear that the ring $B$ is integral over $A$.

For the generator $\mu$, there are two cases: $u\in \overline{\mathbb{Q}}$ or $u\not\in \overline{\mathbb{Q}}$.

\emph{Step 5.} Let $\mu$ be irreducible  in $\mathcal{O}_{K}$. Prove that there exists the case that $u\in \overline{\mathbb{Q}}.$
Hypothesize $$\mu\in \overline{\mathbb{Q}(t_{1},t_{2},\cdots,t_{n})}\setminus \overline{%
\mathbb{Q}}.$$

Put
$$
\mathfrak{P}_{0}\triangleq(\mathcal{O}_{K}\setminus \mathfrak{P})^{-1}(\mu)\subseteq (\mathcal{O}_{K}\setminus \mathfrak{P})^{-1}\mathfrak{P}.
$$
Then $\mathfrak{P}_{0}$ is a prime ideal
of the ring $B$ since $(\mu)$ is a prime ideal of $\mathcal{O}_{K}$.

As $B$ is integral over $A$, from \emph{Going-Down Theorem} we have a polynomial
$$q(t_{1},t_{2},\cdots,t_{n})\in \mathbb{Z}[t_{1},t_{2},\cdots,t_{n}]$$
such that $$\mathfrak{p}_{0}=\mathfrak{P}_{0}\bigcap A$$
where $$\mathfrak{p}_{0}\triangleq(\mathbb{Z}(t_{1},t_{2},\cdots,t_{n})\setminus \mathfrak{p})^{-1}(q(t_{1},t_{2},\cdots,t_{n}))$$
is a prime ideal of $A$ and $q(X_{1},X_{2},\cdots,X_{n})$ is a polynomial with coefficients in $\mathbb{Z}$ such that $q(X_{1},X_{2},\cdots,X_{n})$ is irreducible both in $A$ and in $\mathbb{Z}[X_{1},X_{2},\cdots,X_{n}]$.

Now prove $$\mu=q(t_{1},t_{2},\cdots,t_{n})$$ or $$\mu=-q(t_{1},t_{2},\cdots,t_{n}).$$

In fact, according to \emph{Step 1}, $Spec(B)/Spec(A)$ is \'{e}tale by the
morphism induced from the inclusion map; it follows that $\mathfrak{p}_{0}$ generates $\mathfrak{P}_{0}$ in $B$. Hence we have an element of the form $$\frac{a(t_{1},t_{2},\cdots,t_{n},\mu)}{b(t_{1},t_{2},\cdots,t_{n})}\in B$$ such that
\begin{equation*}
\mu
=\frac{a(t_{1},t_{2},\cdots,t_{n},\mu)}{b(t_{1},t_{2},\cdots,t_{n})}\cdot q(t_{1},t_{2},\cdots,t_{n})
\end{equation*}
holds, where  $a(X_{1},X_{2},\cdots,X_{n},Y)$ and $b(X_{1},X_{2},\cdots,X_{n})$ are polynomials over $\mathbb{Z}$ in several indeterminates.

Put
\begin{equation*}
a(X_{1},X_{2},\cdots,X_{n},Y)=\sum _{0\leq i_{1},i_{2},\cdots, i_{n}\leq l, 0\leq j\leq m}a_{i_{1}i_{2}\cdots i_{n}j}X_{1}^{i_{1}}X_{2}^{i_{2}}\cdots X_{n}^{i_{n}}Y^{j}
\end{equation*}
such that each coefficient $a_{i_{1}i_{2}\cdots i_{n}j}$ is contained in $ \mathbb{Z}$.

We have
$$\mu
=\frac{c(t_{1},t_{2},\cdots,t_{n})+d(t_{1},t_{2},\cdots,t_{n})\cdot\mu}{b(t_{1},t_{2},\cdots,t_{n})}\cdot q(t_{1},t_{2},\cdots,t_{n})$$
where
$$c(X_{1},X_{2},\cdots,X_{n})=\sum _{0\leq i_{1},i_{2},\cdots, i_{n}\leq l,  j=0}a_{i_{1}i_{2}\cdots i_{n}j}X_{1}^{i_{1}}X_{2}^{i_{2}}\cdots X_{n}^{i_{n}};$$
$$d(X_{1},X_{2},\cdots,X_{n})=\sum _{0\leq i_{1},i_{2},\cdots, i_{n}\leq l,  j=1}a_{i_{1}i_{2}\cdots i_{n}j}X_{1}^{i_{1}}X_{2}^{i_{2}}\cdots X_{n}^{i_{n}}.$$

It is seen that
$$c(X_{1},X_{2},\cdots,X_{n})=0$$
holds since in $\mathcal{O}_{K}$ we have
\begin{equation*}
\begin{array}{l}
b(t_{1},t_{2},\cdots,t_{n})\cdot\mu\\

=c(t_{1},t_{2},\cdots,t_{n})q(t_{1},t_{2},\cdots,t_{n})+d(t_{1},t_{2},\cdots,t_{n})
q(t_{1},t_{2},\cdots,t_{n})\cdot\mu.
\end{array}
\end{equation*}

Then
$$\frac{d(t_{1},t_{2},\cdots,t_{n})\cdot q(t_{1},t_{2},\cdots,t_{n})}{b(t_{1},t_{2},\cdots,t_{n})}=1.$$

Immediately, we have $$b(t_{1},t_{2},\cdots,t_{n})\mid d(t_{1},t_{2},\cdots,t_{n})$$ as elements in $\mathbb{Z}[t_{1},t_{2},\cdots,t_{n}]$. Set $$f(t_{1},t_{2},\cdots,t_{n})\triangleq\frac{d(t_{1},t_{2},\cdots,t_{n})}{b(t_{1},t_{2},\cdots,t_{n})}.$$

As $\mu=f(t_{1},t_{2},\cdots,t_{n})\cdot q(t_{1},t_{2},\cdots,t_{n})\cdot\mu,$ we must have
$$f(t_{1},t_{2},\cdots,t_{n})\cdot q(t_{1},t_{2},\cdots,t_{n})=1,$$
which will be in contradiction.

Hence, $$\mu\in \overline{\mathbb{Q}}.$$

\emph{Step 6.} Let $\mu\in \overline{\mathbb{Q}}$  be irreducible  in $\mathcal{O}_{K}$. We have
\begin{equation*}
\mathfrak{p}_{0}=\mathfrak{P}_{0}\bigcap \mathbb{Z}
\end{equation*}
where $\mathfrak{p}_{0}=(p)$ and $\mathfrak{P}_{0}=(\mu)$ are prime ideals of $\mathbb{Z}$ and $\mathbb{Z}[\mu]$, respectively.

It is seen that $\mathfrak{p}_{0}$ generates $\mathfrak{P}_{0}$ in the localization $(\mathbb{Z}[\mu]\setminus \mathfrak{P}_{0})^{-1}\mathbb{Z}[\mu]$ of the ring $\mathbb{Z}[\mu]$ according to preliminary operations on ideals since $\mathfrak{p}$ generates $\mathfrak{P}$ in the ring $B$ and the residue
field $B/\mathfrak{P}$ is a finite separable extension of $A/\mathfrak{p}$
from the assumption.

Let $p$ run through all prime numbers in $\mathbb{Z}$. By definition for unramified it is seen the subfield $K_{0}=\mathbb{Q}[\mu]$
is unramified over $\mathbb{Q}$, or
equivalently, $K_{0}=\mathbb{Q}[\mu]$ is arithmetically  unramified over $\mathbb{Q}$.

It follows that we must have
$
\mu=p
$
since $\mathbb{Q}$ has a trivial maximal unramified extension (see \cite{neu}). Hence, $$ \mu\in \mathbb{Z}[t_{1},t_{2},\cdots,t_{n}].$$

Now from \emph{Steps 3,6}, it is immediate that
\begin{equation*}
K=\mathbb{Q}(t_{1},t_{2},\cdots,t_{n})
\end{equation*}
holds; then
\begin{equation*}
\mathbb{Q}(t_{1},t_{2},\cdots,t_{n})=\mathbb{Q}(t_{1},t_{2},\cdots,t_{n})^{au}.
\end{equation*}
This completes the proof.
\end{proof}

\section{An application to the \'{e}tale fundamental group}

It will be seen that the maximal arithmetically unramified extension can serve as a computation of the \'{e}tale fundamental group of an arithmetic scheme (see \cite{an2,an4} for a general case).
In this section there will be an application of  \emph{Theorem 4.1} to the the \'{e}tale fundamental group of the affine scheme of the ring of integers of a purely transcendental extension over $\mathbb{Q}$.

\subsection{Statement of the theorem}

For a connected scheme $Z$, let $\pi
_{1}^{et}\left( Z,s\right) $ denote the \'{e}tale fundamental group of $Z$ over a
geometric point $s$ of $Z$. See \cite{f-k,sga1,mln,sz} for the definition and properties of \'{e}tale fundamental groups.

\begin{theorem}
$Spec(\mathbb{Z}[t_{1},t_{2},\cdots,t_{n}])$ has a trivial \'{e}tale fundamental group, i.e., there is an isomorphism
$$
\pi _{1}^{et}\left( Spec(\mathbb{Z}[t_{1},t_{2},\cdots,t_{n}]),s\right) \cong \{0\}
$$
of groups, where $s$ is a geometric point of $Spec(\mathbb{Z}[t_{1},t_{2},\cdots,t_{n}])$.
\end{theorem}

We will give the proof of \emph{Theorem 5.1} in \S 5.4.

\subsection{Recalling the \'{e}tale fundamental group of a normal scheme}

Let $X$ be a connected normal scheme surjectively over $\mathbb{Z}$ of finite type.

Let $\Omega=\overline{k(X)}$ be an algebraic closure and $s$  a geometric point of $X$ over $\Omega$.
As in \cite{an2,an4}, we put
\begin{itemize}
\item $\left[X; \Omega \right]_{et}\triangleq$ the set of the pointed covers
$(Y,s_{Y})$ of $(X,s)$, where $Y$ is an irreducible finite \'{e}tale
Galois cover of $X$, $s_{Y} $ is a geometric point of $Y$ over $s_{\xi}$,
and two such covers $(Y,s_{Y}),(Z,s_{Z}) $ of $(X,s)$ are identified
with each other in $\left[X; \Omega \right]_{et}$ if there is an isomorphism
$f$ of $Y$ onto $Z$ over $X$ with $f(s_{Y})=s_{Z}$.
\end{itemize}

See \cite{f-k,sga1,sz} for definition of finite \'{e}tale Galois covers. For brevity, we write $Y=(Y,s_{Y})$ for a $%
(Y,s_{Y})\in \left[ X;\Omega \right] _{et}$.

For any $X_{\alpha},X_{\beta}\in \left[X; \Omega \right]_{et}$, we say
\begin{equation*}
X_{\alpha}\leq X_{\beta}
\end{equation*}
if and only if $X_{\beta}$ is a finite \'{e}tale Galois cover over $%
X_{\alpha}.$ Then $\left[X; \Omega \right]_{et} $ is a directed set with the
partial order $\leq$.

Set

\begin{itemize}
\item $k(X; \Omega)^{un}\triangleq {{\lim}_{\rightarrow_{Z\in {\left[X;
\Omega \right]_{et}}}}}{\lambda_{Z}(k(Z))} $, i.e., $k(X; \Omega)^{un}$ is
the direct limit of the direct system of function fields $k(Z)$ indexed by $%
\left[X; \Omega \right]_{et}$, where each $\lambda_{Z}:k(Z)\to \Omega$ is
the $k(X)$-embedding of fields.
\end{itemize}

$k(X; \Omega)^{un}$ is said to be the \textbf{maximal formally unramified extension} of the arithmetic variety $X$ in $\Omega$ (see \cite{an2,an4}).

It is easily seen that $k(X; \Omega)^{un}$ is an algebraic Galois extension of $k(X)$.

Here there is a computation of the \'{e}tale fundamental group of a connected normal
scheme.

\begin{proposition}
The \'{e}tale fundamental group of $X$ is canonically isomorphic to the Galois group of the maximally formally unramified extension of the function field $k(X)$. That is, there is an isomorphism
\begin{equation*}
\pi _{1}^{et}\left( X,s\right) \cong Gal\left( {k(X; \Omega})^{un}/k\left(
X\right) \right)
\end{equation*}
of groups in a natural manner.
\end{proposition}

\begin{proof}
See \cite{an2,an4,f-k,sga1,mln,sz}.
\end{proof}

\subsection{A computation of the maximally formally unramified extension}

Let $X$ be a connected normal scheme surjectively over $\mathbb{Z}$ of finite type. The function field $k(X)$ of $X$ is an arithmetic function field defined as in \S 2.

For the function field $k(X)$, put $$k(X;\Omega)^{au}=(k(X);\Omega)^{au}.$$ Here there is a computation of the maximal formally
ramified extension.

\begin{proposition}
The maximal arithmetically
unramfied extension $k(X)^{au}$ is equal to the maximal formally ramified
extension $k(X)^{un}$.  That is, we have
\begin{equation*}
k(X;\Omega)^{au}= k(X;\Omega)^{un}
\end{equation*}
as subfields of a given algebraic closure $\Omega$ of the function field $k(X)$.
\end{proposition}

\begin{proof}
(See \cite{an2,an4})
Write $L=k(X)$. Let $\mathcal{O}_{L}$ be the ring  of algebraic integers of the arithmetic function field $L$ over the specified variables. We have
\begin{equation*}
L=k(Spec(\mathcal{O}_{L});
\end{equation*}
\begin{equation*}
(L;\Omega)^{un}=(k(Spec(\mathcal{O}_{L}));\Omega)^{un}.
\end{equation*}

Then it suffices to prove
\begin{equation*}
(k(Spec(\mathcal{O}_{L}));\Omega)^{un}=(L;\Omega)^{au}.
\end{equation*}
In fact, we have
\begin{equation*}
(L;\Omega)^{au}= {{\lim}_{\rightarrow_{Z\in {\left[ L \right]%
(\Omega)^{au}_{0}}}}}{\ \lambda_{Z}(Z)};
\end{equation*}
\begin{equation*}
(k(Spec(\mathcal{O}_{L}));\Omega)^{un}= {{\lim}_{\rightarrow_{X\in {\left[%
Spec(\mathcal{O}_{L}); \Omega \right]_{et}}}}}{\lambda_{X}(k(X))}.
\end{equation*}

On the other hand, put
\begin{equation*}
\Sigma=\{k(X)\mid X\in {\left[Spec(\mathcal{O}_{L});
\Omega \right]_{et}}\}.
\end{equation*}
Then $\Sigma$ is a directed set, where the partial order is given by set inclusion.

As $Spec(\mathcal{O}_{L})$ is a normal scheme, it is easily seen that each $X\in {%
\left[Spec(\mathcal{O}_{L}); \Omega \right]_{et}}$ is normal from
preliminary facts on \'{e}tale covers (see \cite{sga1,mln}). Then $\left[ L %
\right](\Omega)^{au}_{0}$ is a cofinal directed subset of $\Sigma$.
Hence,
\begin{equation*}
\begin{array}{l}
(k(Spec(\mathcal{O}_{L}));\Omega)^{un} \\
={{\lim}_{\rightarrow_{X\in {\left[Spec(\mathcal{O}_{L}); \Omega \right]_{et}%
}}}}{\lambda_{X}(k(X))} \\
= {{\lim}_{\rightarrow_{Z\in {\Psi[Spec(\mathcal{O}_{L})]}}}}{\lambda_{Z}(Z)}
\\
={{\lim}_{\rightarrow_{Z\in {\left[ L \right](\Omega)^{au}_{0}}}}}{\
\lambda_{Z}(Z)} \\
=(L;\Omega)^{au}.%
\end{array}%
\end{equation*}
This completes the proof.
\end{proof}

\subsection{Proof of the main theorem}

Now we can prove \emph{Theorem 5.1}.

\begin{proof}
(\textbf{Proof of \emph{Theorem 5.1}}) Let $\Omega$ be an algebraic closure  of the arithmetic function field $k( Spec(\mathbb{Z}[t_{1},t_{2},\cdots,t_{n}])=\mathbb{Q}(t_{1},t_{2},\cdots,t_{n})$. Take a geometric point $s$ of $Spec(\mathbb{Z}[t_{1},t_{2},\cdots,t_{n}])$ over $\Omega$.
Then there are natural isomorphisms
\begin{equation*}
\begin{array}{l}
\pi _{1}^{et}\left( Spec(\mathbb{Z}[t_{1},t_{2},\cdots,t_{n}]),s\right)\\

\cong Gal\left( {k( Spec(\mathbb{Z}[t_{1},t_{2},\cdots,t_{n}]); \Omega})^{un}/k\left(
X\right) \right)\\

\cong Gal\left( {k( Spec(\mathbb{Z}[t_{1},t_{2},\cdots,t_{n}]); \Omega})^{au}/k\left(
X\right) \right)\\

\cong Gal\left( {\mathbb{Q}(t_{1},t_{2},\cdots,t_{n})}^{un}/{\mathbb{Q}(t_{1},t_{2},\cdots,t_{n})} \right)\\

\cong Gal\left( {\mathbb{Q}(t_{1},t_{2},\cdots,t_{n})}/{\mathbb{Q}(t_{1},t_{2},\cdots,t_{n})} \right)\\

\cong \{0\}
\end{array}
\end{equation*}
between groups
from \emph{Propositions 5.2-3, Theorem 4.1}.
\end{proof}

\end{document}